\documentclass[a4paper,parskip=half*]{scrartcl}

\usepackage[english]{babel}
\usepackage[utf8]{inputenc}

\usepackage{amsmath}
\usepackage{amsthm}
\usepackage{amssymb} 

\usepackage{mathtools}
    \mathtoolsset{centercolon}

\usepackage{xcolor}

\usepackage{enumitem}  
\usepackage{hyperref}
\usepackage[capitalize]{cleveref}

\usepackage[alphabetic]{amsrefs}

\usepackage{tikz}

\theoremstyle{plain}
  \newtheorem{mainthm}{Theorem}
  \newtheorem{thm}{Theorem}[section]
  \newtheorem{prop}[thm]{Proposition}
  \newtheorem{lemma}[thm]{Lemma}
  \newtheorem{cor}[thm]{Corollary}

\theoremstyle{definition}
  \newtheorem{defn}[thm]{Definition}
  \newtheorem{nota}[thm]{Notation}
  \newtheorem{remark}[thm]{Remark}
  \newtheorem{example}[thm]{Example}

\DeclareMathOperator{\Aut}{Aut}
\DeclareMathOperator{\Sym}{Sym}

\DeclareMathOperator{\SL}{SL}

\newcommand{\ZZ}{\mathbb{Z}}
\newcommand{\NN}{\mathbb{N}}
\newcommand{\FF}{\mathbb{F}}

\newcommand{\n}{\mathbf{n}}
\newcommand{\m}{\mathbf{m}}

\newcommand{\w}{\mathbf{w}}

\DeclarePairedDelimiter\gen{\langle}{\rangle}

\newcommand{\mtr}[4]{\begin{pmatrix} {#1} & {#2} \\ {#3} & {#4} \end{pmatrix}} 
\newcommand{\smtr}[4]{\left(\begin{smallmatrix} {#1} & {#2} \\ {#3} & {#4} \end{smallmatrix}\right)} 

\newcommand{\Defn}[1]{\textcolor{blue}{\textit{#1}}}

\newenvironment{axioms}[2][]%
{\begin{enumerate}[label={(#2\arabic*)},ref={#2\arabic*},leftmargin=*,#1]}%
{\end{enumerate}}

\setenumerate[1]{label={(\roman*)}}

\begin{document}

\title{Moufang Twin Trees of prime order}

\author{
Matthias Grüninger\thanks{\url{matthias.grueninger@mathematik.uni-wuerzburg.de}}
\and Max Horn\thanks{\url{max.horn@math.uni-giessen.de}}
\and Bernhard Mühlherr\thanks{\url{bernhard.muehlherr@math.uni-giessen.de}}
}

\maketitle

\begin{abstract}
We prove that the unipotent horocyclic group of a Moufang twin tree of prime order
is nilpotent of class at most 2.
\end{abstract}

\section{Introduction}

The classification of spherical buildings asserts that each irreducible spherical building
of rank at least 3 is of algebraic origin.
By this we mean that it is the building of a classical group, or a semi-simple algebraic group,
or some variation thereof.
In the rank 2 case, this is no longer true; in particular,
there are free constructions of generalized polygons.
(Generalized polygons are precisely the spherical buildings
of rank 2.) In order to characterize the generalized polygons of algebraic origin,
Tits introduced the Moufang condition for spherical buildings in the 1970s \cite{Ti77}. This condition is automatically
satisfied for irreducible spherical buildings of rank at least 3. The Moufang polygons 
were classified in \cite{TW}. It follows from  this classification   
that the Moufang condition characterizes indeed the generalized polygons of algebraic origin.

In the late 1980s Ronan and Tits introduced twin buildings, which were motivated by 
the theory of Kac-Moody groups. 
Twin buildings are generalizations of spherical buildings. For the latter there is
a natural opposition relation on the set of its chambers due the existence
of a unique longest element in the finite Weyl group. Many important results about spherical
buildings (e.g.\ their classification in higher rank) rely on the presence of the opposition
relation. For Kac-Moody groups over fields there is
a natural notion of \Defn{opposite Borel groups}, even if its Weyl-group is infinite. The idea
underlying the definition of twin buildings is to translate this algebraic fact into combinatorics.
Roughly speaking the existence of an opposition relation for spherical buildings is
axiomatized by the notion of a twinning between two buildings of the
same (possibly non-spherical) type. It turns out that many important notions and concepts
from the theory of spherical buildings have indeed natural analogues in the
context of twin buildings.
In particular, the Moufang condition makes sense for twin buildings. There is
the natural question to which extent
the ``spherical'' results can be generalized to the twin case. In this paper we contribute
to this question in the context of twin trees which are precisely the non-spherical twin buildings
of rank 2.

In view of the main result of \cite{TW} it is natural to ask, whether a classification of Moufang
twin trees is feasible. Our main result can be seen as a major step towards a classification
of Moufang twin trees of prime order (i.e.\ for regular Moufang trees of valency
$p+1$ for some prime $p$). This is of course a rather small subclass of all Moufang twin trees.
As we shall explain below, however,  a classification of all Moufang twin trees seems to 
be out of reach at the moment. In view of our result, there is some hope that a classification
of the locally finite Moufang twin trees might be feasible. The latter are precisely the ones which are interesting for the
theory of lattices in locally compact groups. Indeed, using a construction of Tits in \cite{Ti89} and an important
observation of R\'emy in \cite{RemyCRAS99} one knows that locally finite  Moufang twin trees provide a large class
of lattices in locally compact groups. The examples in this class are irreducible and non-uniform lattices in the full automorphism
group of the product of two locally finite trees. Combining this with a result of Caprace and R\'emy in \cite{CR12} it
turns out that a lot of them are simple as abstract groups. To our knowledge these are the only known
examples of  lattices with these properties. A classification of all locally finite Moufang trees would in particular provide
a better understanding of these examples. 

As already announced in the previous paragraph, we now provide   
more information about  the classification problem for Moufang twin trees.
We recall first that there is the natural question whether the Moufang condition
characterizes the twin trees of algebraic origin,  i.e., the examples
provided by Kac-Moody groups and ``their variations''.
An important invariant of a Moufang twin tree is
a subgroup of its automorphism group which is called its \Defn{unipotent horocyclic group}.
In \cite{Ti89} a general construction of Moufang twin trees is given which uses
this invariant as an essential ingredient.
In \cite[Section 2]{RR06} (see also \cite[Example 67]{AR09}) this
construction was made ``concrete'' for certain parameters in order to
construct ``exotic'' examples of Moufang twin trees with abelian
unipotent horocyclic groups. In this way on gets classes of Moufang twin trees which
one would not like to call of algebraic origin.
Therefore the Moufang condition is not sufficient for characterizing the
algebraic examples. Even worse, in \cite{Ti96} it is shown that there are
uncountably many non-isomorphic twin trees of valency 3. In view of the fact
that for each value of $n$ there is at most one Moufang $n$-gon of valency 3, one has to accept
that the analogy of twin trees and generalized $n$-gons has its limitations.

On the other hand, at present it is not clear whether Moufang twin trees are ``wild'' or whether
there is a powerful structure theory for them. This problem is discussed in \cite{Ti89} and an
abstract construction given therein
provides a tool to obtain all Moufang twin trees. However, this has to be taken with a grain of salt
because the procedure requires some group theoretical parameters. Hence, the construction given in
\cite{Ti89} translates
the classification problem for Moufang twin trees into the problem of classifying these parameters.
The question whether these parameter sets can be classified is also discussed in \cite{Ti89}
and we briefly recall its outcome. First of all it turns out that a classification
of all Moufang twin trees would provide a classification of all Moufang sets. Moufang sets
have been studied intensively over the last 15 years and at present it seems that their classification
is far beyond reach. As the finite Moufang sets are known (see e.g.\ \cite{HKS72})
this difficult problem is not an obstacle if we restrict our attention to locally finite
Moufang trees. However, there is still the problem of describing all possible
commutation relations between the root groups in a Moufang twin tree for a given
pair of Moufang sets. The main result of this paper provides a major step to solve
this problem for Moufang twin trees of prime order.
The commutation relations of a Moufang twin tree are in fact
encoded in its unipotent horocyclic group mentioned before. The first step in our solution
to the problem is to introduce $\ZZ$-systems, in order to axiomatize groups which
are candidates for being the unipotent horocyclic group of a Moufang tree. We then prove
\cref{hauptsatz}, a purely group theoretical result whose statement
requires some preparation.
In order to give at least an idea about its implications for Moufang twin trees,
we state the following consequence of it. As the precise definition of a Moufang twin tree
won't be needed in the paper, we refer to \cite{RT94} for an excellent introduction.

\begin{mainthm} \label{mainthm:detailed}
The unipotent horocyclic group of a Moufang twin tree of
prime order is nilpotent of class at most 2.
\end{mainthm}
As already mentioned, \cref{mainthm:detailed} is a consequence of our purely group theoretical
\cref{hauptsatz}. We indicate how \cref{mainthm:detailed} is deduced from 
\cref{hauptsatz} in \cref{sketch proof mainthm}.

Let us finally point out the following two remarks on \cref{mainthm:detailed}:

\begin{enumerate}
\item
As explained before, the theory of twin buildings was developed in order to
provide the appropriate structures associated to Kac-Moody groups.  Roughly
speaking, the ingredients for defining such a group consist of a generalized
Cartan matrix $A$ and a field $\FF$; the resulting group is denoted by
$G_A(\FF)$.  If the Cartan matrix $A$ is a $2 \times 2$-matrix with
non-positive determinant, then the twin building associated to $G_A(\FF)$ is
a Moufang twin tree of order $|\FF|$ whose automorphism group essentially
coincides with the (adjoint version) of $G_A$.  If $A$ is of affine type
(i.e.\ $\det(A) = 0$) then $G_A(\FF)$ can be realized as a matrix group over
$\FF(t)$. In fact, the examples given in \cref{sec trees and RGD}
correspond to Kac-Moody groups of affine type.  In most cases, however,
$G_A(\FF)$ cannot be realized as a matrix group over a field (see
\cite[Theorem 7.1]{Cap09}). 

\item
We already mentioned that there are uncountably many pairwise non-isomorphic
trivalent Moufang twin trees due to a construction of Tits given in \cite{Ti96}.
In view of our result above, one might hope that Tits' construction provides
all trivalent Moufang twin trees which would give a classification of these objects.
By modifying Tits' ideas we have constructed new examples which show
that this is definitively not the case. Nevertheless we are confident that
a classification of  Moufang twin trees of prime order is feasible. We intend
to come back to this question in a subsequent paper.

\end{enumerate}

\paragraph{Some conventions.}
\begin{itemize}
\item We consider $0$ to be a natural number, i.e., $\NN=\{0,1,2,\ldots\}$.
\item For a prime $p\in\NN$, let $\ZZ_p:=\{0,\dots,p-1\}\subset\NN$ and $\ZZ_p^*:=\{1,\dots,p-1\}\subset\NN$.
Moreover, let $\FF_p:=\ZZ/p\ZZ$ be the prime field of order $p$.
\item For a group $G$, let $G^*:=G\setminus\{1\}$.
\item For $A,B,C\leq G$, set $[A,B,C]:=[[A,B],C]$.
\item For $U\subseteq G$, let $\gen{U}$ be the subgroup of $G$ generated by $U$.
\end{itemize}

\section{Moufang twin trees and RGD-systems}
\label{sec trees and RGD}

As explained in the introduction, the classification problem for Moufang twin trees can be 
translated into a purely group theoretical classification problem. The key notion on
the group theoretic side is that of an RGD-system. We first outline what RGD-systems
are, then review the interplay between Moufang twin trees and RGD-systems.
This will provide the motivation for 
our main result and enable us to state it properly.

In \cite{Ti92} RGD-systems have been introduced by Tits in order to investigate
groups of Kac-Moody type and Moufang buildings.
The abbreviation ``RGD'' stands for ``root group data''.
The axioms for an RGD-system are somewhat technical and
we refer to \cite{AB08} and to \cite{CR09} for the general theory
of RGD-systems.

Here we are only interested in
RGD-systems of type $\tilde{A}_1$, i.e.\ in RGD-systems whose type is the Coxeter system associated with the infinite dihedral group. 
The RGD-axioms given below are adapted to this special case
in which they simplify considerably. This is because the root system $\Phi$ of type $\tilde{A}_1$
has the following concrete description. 

\begin{defn}
For each $z \in \ZZ$ we put $\epsilon_z := 1$ if $z \leq 0$ and $\epsilon_z := -1$ if $z > 0$.
We set $\Phi := \ZZ \times \{ 1,-1\}$,
$\Phi^+ := \{ (z,\epsilon_z) \mid z \in \ZZ \}$ and $\Phi^- := \Phi \setminus \Phi^+$. For $i=0,1$ we define
$r_i \in \Sym(\Phi)$ by $(z,\epsilon) \mapsto (2i-z,-\epsilon)$ and we put $\alpha_i := (i,\epsilon_i)$. Finally,
for $\alpha = (z,\epsilon) \in \Phi$ we put $-\alpha := (z,-\epsilon)$.
\end{defn}

\begin{figure}
  \centering
  \begin{tikzpicture}
	\draw[->,white!60!black,thick] (0,0) -- (5,0);
	\draw[->,white!60!black,thick] (0,0) -- (-5,0);
    \foreach\z in {-4,...,4}{
     \draw (\z,1pt) -- (\z,-3pt) node[anchor=north] {$\z$};
    }
    \foreach\z in {-4,...,0}{
     \fill (\z,1) circle (3pt);
     \draw[thick] (\z,-1) circle (3pt);
    }
    \foreach\z in {1,...,4}{
     \draw[thick] (\z,1) circle (3pt);
     \fill (\z,-1) circle (3pt);
    }
  \end{tikzpicture}
\caption{Root system of type $\tilde{A}_1$; black nodes are positive roots,
white nodes negative roots.}
\end{figure}
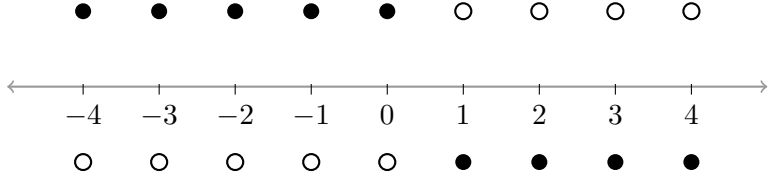

\begin{defn}
An \Defn{RGD-system of type $\tilde{A}_1$} is a triple
$\Pi=(G,(U_{\alpha})_{\alpha \in \Phi},H)$ consisting of a group $G$, a
subgroup $H$ of $G$ and a family $(U_{\alpha})_{\alpha \in \Phi}$ of
subgroups of $G$ (the \Defn{root subgroups}) such that the following holds.
\begin{axioms}{RGD}
\item For all $\alpha \in \Phi$ we have $|U_{\alpha} |>1$.
\item For all $z < z' \in \ZZ$ and all $\epsilon \in \{ 1,-1 \}$ we have
\[ [U_{(z,\epsilon)},U_{(z',\epsilon)}] \in \gen{U_{(n,\epsilon)} \mid z < n < z'}. \]
\item For $i=0,1$ there exists a function $m_i:U_{(i,1)}^* \to G$
such that for all $u \in U_{\alpha_i}^*$ and $\alpha \in \Phi$ we have
\[ m_i(u) \in U_{-\alpha_i} u U_{-\alpha_i}
\quad\text{and }\quad
m_i(u)U_{\alpha}m_i(u)^{-1} = U_{r_i(\alpha)}. \]
Moreover, $m_i(u)^{-1}m_i(v) \in H$ for all $u,v \in U_{\alpha_i}^*$.
\item For $i = 0,1$ the group $U_{-\alpha_i}$ is not contained in $\gen{U_{\alpha} \mid \alpha \in \Phi^+}$.
\item The group $G$ is generated by the family $(U_{\alpha})_{\alpha \in \Phi}$ and the group $H$.
\item The group $H$ normalizes $U_{\alpha}$ for each $\alpha \in \Phi$.
\end{axioms}
\end{defn}

\begin{remark}
We refer to \cite[Definition 7.82 and Subsection 8.6.1]{AB08} for the definition of RGD-systems of arbitrary type. 
In the following discussion ``RGD-system'' shall always
mean ``RGD-system of type $\tilde{A}_1$''.
\end{remark}

\begin{example}[\textbf{The standard example}] \label{ex std}
Let $\FF$ be a field and set 
\[ G:= \SL_2(\FF[t,t^{-1}]) \leq \SL_2(\FF(t)), \qquad
H :=\left\{ \mtr{\lambda}{0}{0}{\lambda^{-1}} \;\middle|\; 0 \neq \lambda \in \FF \right\} \leq G. \]
For each $z \in \ZZ$ we put
\[ U_{(z, 1)} := \left\{ \mtr{1}{\lambda t^z}{0}{1} \;\middle|\; \lambda \in \FF \right\},
   \qquad
   U_{(z,-1)} := \left\{ \mtr{1}{0}{\lambda t^{-z}}{1} \;\middle|\; \lambda \in \FF \right\}.
\]
We point out the following facts:
\begin{enumerate}
\item $\Pi = (G, (U_{\alpha})_{\alpha \in \Phi},H)$ is an RGD-system.
\item Let $U_{++} := \gen{ U_{(z,1)} \mid z \in \ZZ }$. Then
$U_{++} = \left\{ \smtr{1}{f}{0}{1} \;\middle|\; f \in \FF[t,t^{-1}] \right\}$ and in particular
$[U_{(z,1)},U_{(z',1)}] = 1$ for all $z,z' \in \ZZ$.
\item $\gen{ U_{(z,1)},U_{(z,-1)} }$ is isomorphic to $\SL_2(\FF)$ for all $z \in \ZZ$.
\end{enumerate}
\end{example}

\begin{remark}
The following aspect of the standard example is relevant in our context:
Let $\nu$ be a place of $\FF(t)$. Then $\SL_2(\FF(t))$ acts on the Bruhat-Tits tree
$T_{\nu}$ associated with $\nu$. We consider the two rational places $\infty$ and $0$
and set $T_+ := T_{\infty}$ and $T_- := T_0$. It is a fact that there is a twinning $\delta^*$
between $T_+$ and $T_-$ such that $G=\SL_2(\FF[t,t^{-1}])$ acts on the corresponding
Moufang twin tree $T = (T_+,T_-,\delta^*)$ (see \cite{RT94} for details).
Moreover, the unipotent horocyclic group associated with $T$ can be identified with
the group $U_{++}$ defined above.

The interplay between the RGD-system of $\SL_2(\FF[t,t^{-1}])$ and the
twin tree $T$ is actually a special case of a general correspondence
between RGD-systems and Moufang twin trees:
It follows from \cite[Proposition 8.22]{AB08} that each Moufang twin tree $T$
yields an RGD-system $\mathbf{\Pi}(T)$ in a canonical way. Conversely,
for each RGD-system $\Pi$, by \cite[Theorem 8.81]{AB08} there is
a canonical associated twin tree $\mathbf{T}(\Pi)$.
This correspondence is not one-to-one, but it can be made one-to-one
by restricting to RGD-systems of ``adjoint type''.

The following two facts about the correspondence between RGD-systems and
Moufang twin trees are important in our context.
Let $\Pi = (G,(U_{\alpha})_{\alpha \in \Phi},H)$ be an RGD-system and let ${\bf T}(\Pi)$ be the Moufang
twin tree associated with $\Pi$.

\begin{enumerate}
\item As a byproduct of the proof of \cite[Theorem 8.81]{AB08} one observes
that the Moufang twin tree $\mathbf{T}(\Pi)$ is biregular of degree
$(|U_{\alpha_0}|+1,|U_{\alpha_1}|+1)$. In analogy to the theory of
projective planes, we say a tree is of \Defn{order} $q \in \NN$ if it is a
regular tree of degree $q+1$.

\item The group $U_{++} := \gen{U_{(z,1)} \mid z \in \ZZ}$ corresponds
to the unipotent horocyclic group of ${\bf T}(\Pi)$.
\end{enumerate}
\end{remark}

\begin{example}[\textbf{The unitary example.}]  \label{ex unitary}
\cref{mainthm:detailed} in the introduction asserts that the unipotent horocyclic group of a Moufang
twin tree of order $p$ is nilpotent of class at most 2. In the following we want to
provide an example of an RGD-system $\Pi$ which can be realized as a matrix group
and such that the unipotent horocyclic group of ${\bf T}$ is non-abelian.
As this won't be used in the sequel, we omit the details.

Let $\FF$ be field with $\mathrm{char}(\FF)\neq2$. We define the following elements of $\SL_3(\FF(t))$
for $z\in\ZZ$ and $\lambda\in\FF$:
\begin{align*}
x_{2z}(\lambda) &:=
\begin{pmatrix}
  1 & -\lambda t^{z} & (-1)^{z+1} \frac{\lambda^2}{2} t^{2z} \\
    & 1              & \lambda (-t)^{z} \\
    &                & 1
\end{pmatrix}, &
x_{2z+1}(\lambda) &:= 
\begin{pmatrix}
 1 & 0 & (-1)^z \lambda t^{2z+1} \\
   & 1 & 0 \\
   &   & 1
\end{pmatrix},
\\
h(\lambda) &:= \begin{pmatrix} \lambda&&\\&1&\\&&\lambda^{-1} \end{pmatrix}.
\end{align*}
Moreover, we define the following subgroups:
\begin{align*}
  U_{(2z+1,1)} &:= \{ x_{2z+1)}(\lambda) \mid \lambda \in \FF \},
& U_{(2z  ,1)} &:= \{ x_{2z  )}(\lambda) \mid \lambda \in \FF \}, \\
  U_{(2z+1,-1)} &:= U_{(2z+1,1)}^t,
& U_{(2z,-1)} &:= U_{(2z,1)}^t, \\
H &:= \{ h(\lambda) \mid \lambda \in \FF^* \}.
\end{align*}
We set $G := \gen{ U_{\alpha} \mid \alpha \in \Phi }$.
The following can be verified by straightforward calculations.

\begin{itemize}
\item We have $H \leq G$.r
\item $\Pi = (G, (U_{\alpha})_{\alpha \in \Phi},H)$ is an RGD-system.
\item Each $U_{\alpha}$ is isomorphic to the additive group of $\FF$.
\item $U_{++} := \gen{ U_{z,1} \mid z \in \ZZ }$ is non-abelian. Indeed, while
the root groups $U_{2z+1,1}$ are central, we have for $z,z'\in\ZZ$ and $\lambda,\mu\in\FF$ that
\begin{align*}
 [ x_{4z}(\lambda),\ x_{4z'+2}(\mu) ] &= x_{2z+2z'+1}(2\lambda\mu),\\
 [ x_{4z+2}(\lambda),\ x_{4z'}(\mu) ] &= x_{2z+2z'+1}(-2\lambda\mu),\\
 [ x_{4z}(\lambda),\ x_{4z'}(\mu) ] &=  [ x_{4z+2}(\lambda),\ x_{4z'+2}(\mu) ] = 1_G.
\end{align*}
\end{itemize}
\end{example}

\section{The main result}

As consequence of the discussion in the previous section, we conclude that the classification
of Moufang twin trees of prime order $p$ is equivalent to the classification of RGD-systems
in which all $U_{\alpha}$ have order $p$. The Moufang sets of cardinality $p+1$ are
classified. Thus, the main obstacle remaining in the classification of Moufang twin trees of prime
order 
is the classification of the possible commutation relations.
In order to make this more concrete, first consider the
following basic observation about RGD-systems.  

\begin{lemma} \label{Zsystemmotivation}
Let $\Pi=(G,(U_{\alpha})_{\alpha \in \Phi},H)$ be an RGD-system. Let $X_n := U_{(n,1)}$
for each $n \in \ZZ$ and $X := \gen{X_n \mid n \in \ZZ}$. 
Then the following hold.
\begin{enumerate}
\item For all $n \leq m \in \ZZ$ the product map
$X_n \times X_{n+1} \times \dots \times X_m \to \gen{X_i \mid n \leq i \leq m}$
is a bijection.
\item There exists $t \in \Aut(X)$ such that $t(X_n) = X_{n+2}$
for all $n \in \ZZ$.
\end{enumerate}
\end{lemma}

\begin{proof}
Assertion (i) follows from Assertion (i) of Corollary 8.34 in \cite{AB08}. Let $i =0,1$. Using the
function $m_i$ from (RGD3), we can construct $s_i \in G$ such that
$U_{\alpha}^{s_i} = U_{r_i(\alpha)}$ for all $\alpha \in \Phi$.
Then the mapping $t:X \to X, x \mapsto x^{s_0s_1}$ has the required properties.
\end{proof}

As we are dealing with Moufang twin trees of prime order, we have to consider
RGD-systems in which all the $U_{\alpha}$ have order $p$ for some prime number $p$.
Let $\Pi=(G,(U_{\alpha})_{\alpha \in \Phi},H)$ be such an RGD-system,
and let $X$, $(X_n)_{n \in \ZZ}$ and $t$ be as in the previous lemma.
By choosing $1 \neq x_i \in U_{(i,1)}$ for $i=0,1$ and setting
$x_{2n} := t^n(x_0)$ and $x_{2n+1} := t^n(x_1)$, we obtain
a pair $(X,(x_n)_{n \in \ZZ})$ conforming to the following definition.

\begin{defn} \label{def zsys}
Let $p$ be a prime.
A \Defn{$\ZZ$-system (of order $p$)} is a pair $(X,(x_n)_{n \in \ZZ})$ consisting of a group $X$ and
a family $(x_n)_{n \in \ZZ}$ of elements in $X$ such that the following conditions
are satisfied.
\begin{axioms}{ZS}
\item $X = \gen{x_n \mid n \in \ZZ}$.
\item For all $n \leq m \in \ZZ$ the group $\gen{x_k \mid n \leq k \leq m}$ is of order $p^{m-n+1}$.
\item There exists an automorphism $t$ of $X$ such that $t(x_n) = x_{n+2}$
for all $n \in \ZZ$.
\end{axioms}
\end{defn}

\begin{example}
Let $p$ be a prime and let $\FF:=\FF_p$.
\begin{enumerate}
\item 
Let everything be as in \cref{ex std}. For $n\in\ZZ$ let $u_n:=\smtr{1}{t^n}{0}{1}$. Then
$(U_{++},(u_n)_{n \in \ZZ})$ is a $\ZZ$-system of order $p$.
Indeed, the map
\[ U_{++}\to U_{++},\ g \mapsto g^\sigma,\quad\text{ where }\quad
\sigma:=\mtr{t^{-1}}{0}{0}{t}\]
is an automorphism of $U_{++}$ which maps $u_n$ to $u_{n+2}$ for all $n\in\ZZ$.

\item
Let everything be as in \cref{ex unitary}. For $n\in\ZZ$ let $u_n:=x_{n}(1_{\FF_p})$.
Then $(U_{++},(u_n)_{n \in \ZZ})$ is a $\ZZ$-system of order $p$.
Indeed, the map
\[ U_{++}\to U_{++},\ g \mapsto g^\sigma,\quad\text{ where }\quad
\sigma:=\begin{pmatrix} t^{-1}&&\\&1&\\&&-t \end{pmatrix}\]
is an automorphism of $U_{++}$ which maps $u_n$ to $u_{n+2}$ for all $n\in\ZZ$.
\end{enumerate}
\end{example}

We already mentioned in the introduction that Tits gave a construction of
uncountably many pairwise non-isomorphic trivalent twin trees. The idea
behind his construction can be generalized to produce uncountably many
non-isomorphic $Z$-systems of order $p$ for each prime $p$. It is conceivable
that only very few of them can be realized as matrix groups. In a sense,
Axiom (ZS3) requires an analogue of the conjugation by a diagonal matrix in the
non-linear context.

We are now in the position to state our main result, which we prove in \cref{sect hauptbeweis}.

\begin{thm} \label{hauptsatz}
Let $(X,(x_n)_{n \in \ZZ})$ be a $\ZZ$-system of prime order.
Then $X$ is nilpotent of class at most 2.
\end{thm}

\begin{remark}[Sketch of the proof of \cref{mainthm:detailed}]
\label{sketch proof mainthm}
Let ${\bf T}$ be a Moufang twin tree of order $p$ and
let $\Pi(T) = (G, (U_{\alpha})_{\alpha \in \Phi},H)$ be the RGD-system associated
with $T$. As $T$ is of order $p$, each $U_{\alpha}$ has order $p$.
By \cref{Zsystemmotivation} we therefore obtain a $\ZZ$-system $(X,(x_n)_{n \in \ZZ})$
of order $p$, and the unipotent horocyclic group of $T$ coincides with $X$.
Thus \cref{mainthm:detailed} is a consequence of  \cref{hauptsatz}
\end{remark}

\section{$\ZZ$-systems}
\label{ZS-Section}

For the rest of this paper, we assume that $p$ is a prime and that
$\Theta = (X,(x_n)_{n \in \ZZ})$ is a $\ZZ$-system of order $p$,
together with an automorphism $t\in\Aut(X)$ as in (ZS3), the \Defn{shift automorphism} of $\Theta$.
In the following lemma we collect some basic properties of $\ZZ$-systems.

\begin{defn}
For $n\leq m\in\ZZ$, we set
\begin{align*}
X_{n,m} &:=\gen{x_k \mid n \leq k \leq m}, &
X_{-\infty,m}&:=\gen{x_k \mid k \leq m}, &
X_{n,\infty}&:=\gen{x_k \mid n \leq k}.
\end{align*}
\end{defn}

\begin{lemma}
The following statements are true.
\begin{axioms}[start=4]{ZS}
\item For each $n \in \ZZ$ we have $x_n^p = 1 \neq x_n$.
\item For $n < m \in \ZZ$ we have $[x_n,x_m] \in X_{n+1,m-1}$.
\item For each $x \in X^*$ there exist $n \leq m \in \ZZ$ and
$e_n,\ldots,e_m \in \ZZ_p$ such that $x = x_n^{e_n}\cdots x_m^{e_m}$,
and both $e_n\neq0$ and $e_m\neq 0$.
Moreover, $n,m,e_n,\ldots,e_m$ are uniquely determined by $x$.
\end{axioms}
\end{lemma}

\begin{proof}
(ZS4) is immediate from (ZS2) with $m=n$.
Now recall that a subgroup of index $p$ in a finite $p$-group is normal.
Hence for any $n\leq m\in\ZZ$, we obtain the following normal series,
where each group has index $p$ in the preceding one:
\[ X_{n,m}\rhd X_{n+1,m}\rhd \dots \rhd X_{m,m}\rhd 1.\]
Thus $x_n,\dots,x_m$ form a polycyclic generating sequence of $X_{n,m}$.
Then (ZS6) follows.
From this it also follows that $X_{n,m}'\leq X_{n+1,m}$.
By a symmetric argument $X_{n,m}'\leq X_{n,m-1}$ and hence (ZS5) follows.
\end{proof}

\begin{defn}
Let $x \in X^*$. By (ZS6) there exist unique $n \leq m \in\ZZ$
and $e_n,\ldots,e_m \in \ZZ_p$ such that
$e_n\neq0\neq e_m$ and $x = x_n^{e_n}\cdots x_m^{e_m}$.
This is the \Defn{normal form} of $x$, and we set
\[
\n(x) := n, \quad
\m(x) := m.
\]
The \Defn{width} of $x\in X^*$ is $\w(x) := m-n+1$. Additionally
we set $\w(1):=0$, $\n(1):=\infty$ and $\m(1):=-\infty$.
\end{defn}

Finally we point out some useful direct consequences of (ZS5) and (ZS6),
which we use extensively in the sequel.

\begin{lemma}
Let $x,y\in X^*$.
\begin{enumerate}[ref={\thethm (\roman*)}]
\item \label[lemma]{startindex}
Let $k \in \ZZ$ such that $\n(x)\neq k$. Then $\n(x_kx) = \min(k, \n(x))$.
\item \label[lemma]{cancel start} If $\n(x) = \n(y)$, then 
there is $\lambda\in\ZZ_p^*$ such that
$\n(x)<\n(y^\lambda x)$ and $\w(y^\lambda x) < \max(\w(x),\w(y))$.
\item \label[lemma]{cancel end} If $\m(x) = \m(y)$, then 
there is $\lambda\in\ZZ_p^*$ such that
$\m(y^\lambda x)<\m(x)$ and $\w(y^\lambda x) < \max(\w(x),\w(y))$.
\item \label[lemma]{cancel power} $\w(x^p)<\w(x)$.
\end{enumerate}
\end{lemma}

\section{Abelian $\ZZ$-systems}

In this section we establish a criterion for proving that a $\ZZ$-system is abelian,
stated as \cref{abelian<=>single-shift-inv}.

\begin{defn}
The \Defn{lower cutoff} of $\Theta$ is defined as 
\[
\ell(\Theta) := 
\begin{cases}
\infty & \text{ if } X \text{ is abelian}, \\
\min \{ |m-n| \mid [x_n,x_m] \neq 1 \} & \text{ if } X \text{ is non-abelian}.
\end{cases}
\]
\end{defn}

Recall that by (ZS3) there is an automorphism $t$ of $X$ mapping $x_n$ onto $x_{n+2}$ for all $n \in \ZZ$.

\begin{lemma} \label{MaxProposition5.4(1)}
Let $X$ be non-abelian and let $n := \ell(\Theta)$ be the lower cutoff of $\Theta$.
If $[x_0,x_n] \neq 1$, then $[x_1,x_{n+1}] = 1$; if $[x_1,x_{n+1}] \neq 1$, then $[x_0,x_n] = 1$.
\end{lemma}

\begin{proof}
Suppose $w:=[x_0,x_n] \neq 1$. As $n$ is the lower cutoff of $\Theta$,
the subgroup $X_{-(n-1),n-1}$ centralizes $x_0$.
Similarly $X_{1,2n-1}$ centralizes $x_n$.
Thus, for $0 \leq j <n$,
\[ [x_0,x_{n+j}] \in X_{1,n+j-1}\leq X_{1,2n-1},
\quad\text{ implying }\quad [[x_0,x_{n+j}],x_n] = 1. \]
Since $j<n$ we have also $[x_n,x_{n+j}] = 1$, hence $[[x_n,x_{n+j}],x_0] = 1$.
Then the Three Subgroup Lemma (see e.g.\ 
\cite[5.1.10]{Rob96})
implies $[w,x_{n+j}] = [[x_0,x_n],x_{n+j}] = 1$.

Let $i := \n(w)$. As $w \in X_{1,n-1}$ it follows that $1\leq i<n$ and hence $[w,x_{n+i}] = 1$.
But $w$ can be written as $w = x_i^{e_i}\dots x_{n-1}^{e_{n-1}}$
with $e_i,\dots e_{n-1} \in \ZZ_p$.
Since the lower cutoff is $n$, we have $[x_j,x_{n+1}]=1$ for $i+1\leq j\leq n-1$.
Thus also $[x_i,x_{n+i}] = 1$.

As $[x_{2k},x_{2k+n}] = t^k([x_0,x_n]) = t^k(w) \neq 1$ for all $k \in \ZZ$,
it follows that $i$ must be odd. So there is $m\in\ZZ$ with $i=2m+1$,
therefore $[x_1,x_{n+1}] = t^{-m}([x_i,x_{n+i}]) = t^{-m}(1) = 1$.
This proves the first assertion,
the second follows by a symmetric argument.
\end{proof}

\begin{prop} \label{abelian<=>single-shift-inv}
The following are equivalent:
\begin{enumerate}
\item The group $X$ is abelian.
\item The group $X$ is elementary abelian (i.e.\ abelian and of exponent $p$).
\item The mapping $x_k \mapsto x_{k+1}$ extends to an automorphism of $X$.
\end{enumerate}
\end{prop}

\begin{proof}
By (ZS2), the generators $x_n$ have order $p$.
Thus if $X$ is abelian, then $X$ has exponent $p$.
Thus (i) implies (ii). The converse implication is trivial.
Also that (ii) implies (iii) now is readily verified.

Assume that $X$ is not abelian and let $n:= \ell(\Theta)$.
By \cref{MaxProposition5.4(1)}, $[x_0,x_n] \neq 1$ implies $[x_1,x_{n+1}] = 1$
and $[x_1,x_{n+1}] \neq 1$ implies $[x_0,x_n] = 1$.
Thus, the mapping $x_k \mapsto x_{k + 1}$ does not extend to an automorphism of $X$.
\end{proof}

\section{Shift-invariant subgroups}

In this section we study subgroups of $X$ which are invariant under the shift map $t$.
We prove that such subgroups are close to forming $\ZZ$-systems again. Moreover,
those of infinite index are necessarily abelian.

\begin{defn}
A subgroup $Y\leq X$ is called \Defn{shift-invariant} if $t(Y) = Y$.
We set
\begin{align*}
Y_{even} &:= \{ y \in Y^* \mid \n(y) \in 2\ZZ \},\\
Y_{odd} &:= \{ y \in Y^* \mid \n(y) \in 1+2\ZZ \}.
\end{align*}

For $n\leq m\in\ZZ\cup\{\pm\infty\}$, set $Y_{n,m}:=Y\cap X_{n,m}$.
\end{defn}

\begin{remark}
By shift-invariance of $Y$, we have $t(Y_{n,m})=Y_{n+2,m+2}$.
\end{remark}

\begin{lemma} \label{MaxProposition5.7(1)}
Let $Y \leq X$ be shift-invariant. Then the following are equivalent:
\begin{enumerate}
\item The index of $Y$ in $X$ is finite.
\item Both $Y_{even}$ and $Y_{odd}$ are non-empty.
\end{enumerate}
\end{lemma}

\begin{proof}
Let $Y$ be of finite index in $X$.
Suppose, by contradiction, that $\n(y)$ is even for all $y \in Y^*$.
Then $\n(Y^*) = 2\ZZ$ because $Y$ is shift-invariant. In view of
\cref{startindex}, for each odd integer $m$ we have
\[\n(x_{m}Y) = \{ m \}  \cup \{2k\in2\ZZ\mid 2k<m\}.\]
Hence for any two odd integers $m \neq m'$ we have $x_mY \neq x_{m'}Y$.
Thus we get infinitely many cosets of $Y$, which is a contradiction.

Similarly the assumption that $\n(y)$ is odd for all $y \in Y^*$
leads to a contradiction and hence (i) implies (ii).

For the converse, let $a$ (resp. $b$) be of minimal width in $Y_{even}$ (resp. $Y_{odd}$).
Since $Y$ is shift-invariant, we may assume that $\n(a) = 0$ and $\n(b) = 1$.

We claim that $\m(a)$ and $\m(b)$ have different parity. Suppose that this
is not the case. Then there exists an element $k \in \ZZ$
such that $\m(t^k(a)) = \m(b)$. Using \cref{cancel end}
it follows that there is $\lambda\in\ZZ_p^*$ such that
$y:= b^\lambda t^k(a)$ satisfies either $y\in Y_{even}$ and $\w(y) < \w(a)$,
or $y \in Y_{odd}$ and $\w(y) < \w(b)$. Either case contradicts the minimality
of $a$ resp.\ $b$.

Let $m := \max \{ \w(a),\w(b) \}$. Since $\n(a)=0$ and $\n(b) =1$,
by using \cref{cancel start} and induction, it follows that $X_{-\infty,m}Y  \subseteq  X_{0,m}Y$.
As $\m(a)$ and $\m(b)$ have different parity, one also sees that $X_{0,\infty} Y \subseteq X_{0,m}Y$.
As $X = X_{-\infty,m}X_{0,\infty}$ it follows that $X=X_{0,m}Y$.
Thus $|X:Y|\leq |X_{0,m}| = p^{m+1}$.
\end{proof}

\begin{prop} \label{MaxProposition5.7(3)}
Let $1\neq Y\leq X$ be shift-invariant with $|X:Y| = \infty$,
let $u \in Y^*$ be of minimal width in $Y$ and $y_n := t^n(u)$ for $n \in \ZZ$. 
Then the following hold.
\begin{enumerate}
\item $Y = \gen{y_n \mid n \in \ZZ}$.
\item $(Y,(y_n)_{n \in \ZZ})$ is a $\ZZ$-system.
\item $Y$ is elementary abelian of exponent $p$.
\end{enumerate}
\end{prop}

\begin{proof}
\begin{enumerate}
\item
By shift-invariance of $Y$ we have $y_n\in Y$, thus $U := \gen{y_n \mid n \in \ZZ} \leq Y$.
For $y\in Y$ we will show by induction on $\w(y)$ that $y\in U$, and hence $Y=U$.
If $\w(y)=0$ then $y=1\in U$. So suppose $\w(y)>0$.
Now $|X:Y| = \infty$, therefore $\n(y)$ and $\n(u)$ have the
same parity by \cref{MaxProposition5.7(1)}. Hence there is $k\in\ZZ$
such that
\[ \n(y)=\n(u)+2k=\n(t^k(u))=\n(y_k).\]
Moreover, $\w(y)\geq \w(y_k)=\w(u)$. Thus by \cref{cancel start} there is
$\lambda\in\ZZ_p^*$ such that $\w(y_k^\lambda y)<\w(y)$.
Hence by the induction hypothesis $y_k^\lambda y\in U$.
Since also $y_k\in U$ we get $y\in U$.

\item
(ZS1) follows from Assertion (i).
(ZS3) follows from the fact that $t(Y) = Y$ and $t(y_n) = y_{n+1}$ for all $n \in \ZZ$,
hence $s:=t^2$ is a shift automorphism for $(Y,(y_n)_{n \in \ZZ})$.
It remains to verify (ZS2). Without loss of generality, assume
$\n(u)\in\{0,1\}$ and thus $\n(y_n)\in\{2n,2n+1\}$ for $n\in\ZZ$.

For $n\leq m\in\ZZ$ let $U_{n,m}:=\gen{y_n,\ldots,y_m}\leq X_{2n,\infty}$.
As $\n(y_n) \in\{2n,2n+1\}$, we have $y_n\notin X_{2n+2,\infty}$,
hence $y_n\notin U_{n+1,m}\leq X_{2n+2,\infty}$. 
\Cref{cancel power} implies that $\w(u^p) < \w(u)$.
Since $u$ was of minimal width, we conclude $u^p=1$. Thus $y_n$ has order $p$.
Since $p$ is prime, we get $\gen{y_n}\cap U_{n+1,m}=1$.

Now we claim that  $U_{n,m}=\gen{y_n}U_{n+1,m}$.
To see this, pick $y\in U_{n,m}$. If $\n(y)>\n(y_n)$ then $y\in U_{n+1,m}$.
Otherwise $\n(y)=\n(y_n)$, and then \cref{cancel start} implies that
there is $\lambda\in\ZZ_p^*$ such that $\n(y_n^\lambda y) > \n(y_n)$,
hence $y_n^\lambda y \in U_{n+1,m}$. The claim follows.

But $\gen{y_n}\cap U_{n+1,m}=1$ and $U_{n,m}=\gen{y_n}U_{n+1,m}$
imply $|U_{n,m}|=p\cdot |U_{n+1,m}|$. By induction it follows
that $|U_{n,m}| = p^{m-n+1}$. Thus (ZS2) holds.

\item
By $(ii)$, $(Y,(y_n)_{n \in \ZZ})$ is a $\ZZ$-system. The shift map $t$ of $(X,(x_n)_{n \in \ZZ})$
leaves $Y$ invariant and thus restricts to an automorphism of $Y$ which extends the mapping
$y_k\mapsto y_{k+1}$. The claim thus follows from \cref{abelian<=>single-shift-inv}.
\qedhere
\end{enumerate}
\end{proof}

\begin{lemma} \label{MaxProposition7.1(2)}
Let $Y\leq X$ be shift-invariant with $Y_{even} \neq
\emptyset \neq Y_{odd}$.  Let $a$ (resp. $b$) be of minimal width in
$Y_{even}$ (resp. $Y_{odd}$) such that $\n(a) = 0$ and $\n(b) = 1$.
For $n \in \ZZ$ let $y_{2n} := t^n(a)$ and $y_{2n+1} := t^n(b)$.
Then the following hold:
\begin{enumerate}
\item $Y = \gen{y_n \mid n \in \ZZ}$.
\item If $\w(a) = \w(b)$, then $(Y,(y_n)_{n \in \ZZ})$ is a $\ZZ$-system.
\end{enumerate}
\end{lemma}

\begin{proof}
\begin{enumerate}
\item
By shift-invariance of $Y$ we have $y_n\in Y$, thus $U := \gen{y_n \mid n \in \ZZ} \leq Y$.
For $y\in Y$ we will show by induction on $\w(y)$ that $y\in U$, and hence $Y=U$.
If $\w(y)=0$ then $y=1\in U$. So suppose $\w(y)>0$ and let $n:=\n(y)$.
Then $\w(y)\geq \w(y_n)$. Since $\n(y_n)=n=\n(y)$, by \cref{cancel start} there is
$\lambda\in\ZZ_p^*$ such that $\w(y_n^\lambda y)<\w(y)$.
Hence by the induction hypothesis $y_n^\lambda y\in U$.
Since also $y_n\in U$ we get $y\in U$.

\item
This follows by a similar argument as in the proof of Assertion (ii) in \cref{MaxProposition5.7(3)}.
(Note that we do not make use of this observation in this paper.)
\qedhere
\end{enumerate}
\end{proof}
Combining the previous statements yields the following:

\begin{lemma} \label{shift-inv-gens}
Let $Y$ be a shift-invariant subgroup of $X$. Then there are
elements $a,b\in Y$ such that $Y=\gen{t^k(a), t^k(b) \mid k\in \ZZ}$.
\end{lemma}

\begin{proof}
If $Y$ has finite index in $X$, this follows from
\cref{MaxProposition5.7(1),MaxProposition7.1(2)}.
If $Y$ is trivial, we can choose $a=b=1$.
Finally, if $Y$ is non-trivial but has infinite index, this
follows from \cref{MaxProposition5.7(3)}
\end{proof}

\begin{remark}
We can make the choice of generators $a,b$ unique by requiring that each should either be trivial; or else start at index 0 or 1, be of minimal width amongst all such elements, and have ``lead exponent'' equal to 1.

The resulting generating system is close to being a $\ZZ$-system again. However, the generators are not
necessarily independent anymore; in particular, it can happen that that $a^p=b$.
\end{remark}

\begin{lemma} \label{decompose shift inv}
Let $Y$ be a shift-invariant subgroup of $X$. Then  for every $n\in\ZZ$,
there is $m \in \ZZ$ such that $Y = Y_{-\infty,m} Y_{n,\infty}$.
\end{lemma}

\begin{proof}
Pick $a,b\in Y$ as in \cref{shift-inv-gens}.
Since $Y$ is generated by all shifts of $a$ and $b$, it suffices to choose $m$ large
enough such that $Y_{-\infty,m}$ contains all the shifts of $a$ and $b$ which are not in $Y_{n,\infty}$.
For example, choose $m:= \max\{n + \w(a), n + \w(b) \}$.
\end{proof}

\section{One-sided normal subgroups}

Throughout this section, let $Y$ be a shift-invariant subgroup of $X$.

\begin{nota}
Let $G$ be a group.
The \Defn{normal closure} of $U\subseteq G$ is
$\gen{U}^G := \gen{U^G} = \gen{ g^{-1}Ug \mid g\in G}$.
\end{nota}

\begin{remark}
Recall that a group $G$ is \Defn{locally nilpotent} if every finitely generated
subgroup of $G$ is nilpotent.
Now every finitely generated subgroup $H$ of $X$ is contained in some
$X_{n,m}$ with $n\leq m\in\ZZ$, which is a finite $p$-group by (ZS2).
Hence $H$ is a finite $p$-group, and $X$ is locally nilpotent.
\end{remark}

\begin{lemma} \label{A<=[AK] then A=1}
Let $K$ be nilpotent and $A\leq K$ with $A\leq [A,K]$.
Then $A=1$.
\end{lemma}

\begin{proof}
$K$ is nilpotent, hence its lower central series $K\unrhd [K,K]\unrhd
[K,K,K] \unrhd \dots$ vanishes after finitely many steps.
Since $A\leq K$, also $[A,K,\dots,K]$ eventually vanishes.
From $A\leq [A,K]$ we deduce, by forming the commutator with $K$, that
\[ A\leq [A,K] \leq [A,K,K] \leq \dots \leq 1. \qedhere\]
\end{proof}

\begin{lemma} \label{A-notin-[AU]_U}
Let $G$ be a locally nilpotent group and let $A\leq G$ be finitely generated.
Then $A\leq \gen{[A,G]}^G$ if and only if $|A|=1$.
\end{lemma}

\begin{proof}
The implication starting with $|A|=1$ is obvious.
So suppose $A\leq \gen{[A,G]}^G$ and $A=\gen{a_1,\ldots,a_n}$.
Then for $1\leq i\leq n$, there exist $\ell_i\in\NN$ and elements
$a_{ij}\in A$, $g_{ij},h_{ij}\in G$ such that
\begin{equation} \label{eqn:normal-closure}
 a_i = [a_{i1},g_{i1}]^{h_{i1}} \cdots [a_{i\ell_i},g_{i\ell_i}]^{h_{i\ell_i}}.
\end{equation}
We now define the finitely generated subgroup
\[
H:=\gen{ h_{ij}, g_{ij} \mid 1\leq i\leq n, 1\leq j\leq \ell_i }.
\]
Moreover, we set $K:=\gen{A, H}$, and observe that
\[ A^H := \gen{A}^H =\gen{a^h\mid a\in A, h\in H} \leq K. \]
Since $A$ and $H$ are finitely generated, so is $K$, hence $K$ is nilpotent.
From \Cref{eqn:normal-closure} we then conclude
$A\leq [A^H,H]$
hence
$A^H\leq [A^H,H] \leq [A^H,K]$.
Applying \cref{A<=[AK] then A=1}, we conclude that $A^H=1$.
Hence $A=1$.
\end{proof}

\begin{lemma} \label{lem:100}
Let $n\in\ZZ$. Then there is
$y_{n-1}\in Y_{n-1,\infty}$ such that 
\[ Y_{n-1,\infty} = \gen{ y_{n-1}, Y_{n,\infty} }. \]
Moreover, for any $N\geq \w(y_{n-1})-2$, we have
\[ Y_{n-1,n+N} = \gen{ y_{n-1}, Y_{n,n+N} }.\]
\end{lemma}

\begin{proof}
If $Y_{n-1,\infty}= Y_{n,\infty}$ set $y_{n-1}:=1$ and the first assertion
clearly holds. Otherwise there exists $y_{n-1}\in Y_{n-1,\infty}$ with
$\n(y_{n-1})=n-1$.  Let $y\in Y_{n-1,\infty}$. If $\n(y)\geq n$, then $y\in
Y_{n,\infty}$.  Otherwise, if $\n(y)=n-1$, then by \cref{cancel start} there
is $\lambda\in\ZZ_p^*$ such that $\n(y_{n-1}^\lambda y)\geq n$, hence
$y\in \gen{ y_{n-1}, Y_{n,\infty} }$.
Thus $Y_{n-1,\infty} \leq \gen{y_{n-1}, Y_{n,\infty}}$.
The reverse inclusion is obvious.

The second assertion follows analogously, after observing that $y_{n-1}\in Y_{n-1,n+N}$.
Indeed, $\n(y_{n-1})=n-1$ and $\m(y_{n-1})=\w(y_{n-1})+\n(y_{n-1})-1 = n + (\w(y_{n-1})-2)$.
\end{proof}

\begin{lemma} \label{lem:105}
Let $n\in\ZZ$.
\begin{enumerate}
\item If $Y_{n-1,\infty}\leq\gen{Y_{n,\infty}}^X$, then there is
      $M\in\NN$ such that $Y_{n-1,n+N}\leq\gen{Y_{n,n+N}}^X$ for all $N \geq M$.
\item If $Y_{-\infty,n+1}\leq\gen{Y_{-\infty,n}}^X$, then there is
      $M\in\NN$ such that $Y_{n-N,n+1}\leq\gen{Y_{n-N,n}}^X$ for all $N \geq M$.
\end{enumerate}
\end{lemma}

\begin{proof}
We prove the first case, the second follows by a symmetric argument.
Suppose $Y_{n-1,\infty}\leq\gen{Y_{n,\infty}}^X$.
If $Y_{n-1,\infty}=Y_{n,\infty}$ we are done, as then
$Y_{n-1,n+N}=Y_{n,n+N}$ for all $N\in\NN$.
Otherwise, let $y_{n-1}\neq 1$ be as in \cref{lem:100}.
Since $y_{n-1}\in Y_{n-1,\infty}\leq\gen{Y_{n,\infty}}^X$, there are
$\ell\in\NN$, $a_1,\ldots,a_\ell\in Y_{n,\infty}^*$ and $g_1,\ldots,g_\ell\in X$
such that $y_{n-1} = a_1^{g_1}\cdots a_\ell^{g_\ell}$.
Let $M:=\max\{\m(a_1),\ldots,\m(a_\ell),\m(y_{n-1})\}-n$.
Then $a_1,\ldots,a_\ell\in Y_{n,n+M}$, hence for $N\geq M$ we have 
\[ y_{n-1} \in \gen{Y_{n,n+M}}^X \leq \gen{Y_{n,n+N}}^X.\]
Moreover, by definition 
\[ M\geq \m(y_{n-1}) - n = \m(y_{n-1}) - \n(y_{n-1}) - 1 = \w(y_{n-1}) - 2. \]
Thus for $N\geq M$, \cref{lem:100} yields
\[ Y_{n-1,n+N} = \gen{y_{n-1}, Y_{n,n+N}} \leq \gen{Y_{n,n+N}}^X.
\qedhere
\]
\end{proof}

\begin{lemma}\label{finite cut}\ \nopagebreak
\begin{enumerate}
\item If $Y_{n-1,\infty}\leq\gen{Y_{n,\infty}}^X$ for all $n\in\ZZ$,
then there is $M\in\NN$ with $Y_{-\infty,M}\leq \gen{Y_{0,M}}^X$.
\item If $Y_{-\infty,n+1}\leq\gen{Y_{-\infty,n}}^X$ for all $n\in\ZZ$,
then there is $M\in\NN$ with $Y_{0,\infty}\leq \gen{Y_{0,M}}^X$.
\end{enumerate}
\end{lemma}

\begin{proof}
We prove the first case, the second follows by a symmetric argument.
The hypothesis implies for all $n\in\NN$ that
\[
Y_{n-2,\infty}\leq\gen{Y_{n-1,\infty}}^X
\quad\text{ and }\quad
Y_{n-1,\infty}\leq\gen{Y_{n,\infty}}^X,
\quad\text{ hence }\quad
Y_{n-2,\infty}\leq\gen{Y_{n,\infty}}^X.
 \]
Thus for all $n,k\in\NN$ we have
\[ Y_{n-k,\infty}\leq\gen{Y_{n,\infty}}^X.\]
But this implies $Y\leq\gen{Y_{n,\infty}}^X$. By \cref{shift-inv-gens}, 
there are elements $a,b\in Y$ such that $Y=\gen{t^k(a), t^k(b) \mid k\in \ZZ}$.
As $Y$ is shift-invariant, we may assume $\n(a)=-2$ or $a=1$, and $\n(b)=-1$ or $b=1$.
Then
\[ a,b\in Y_{-2,\infty}\leq \gen{Y_{0,\infty}}^X.\]
By applying \cref{lem:105} twice we deduce the existence of some value
$M\in \NN$ such that $a,b\in \gen{Y_{0,M}}^X$.
But then for all $N\geq M$
\begin{equation} \label{eqn:Y_{-2}-sub-Y_0^X}
Y_{-2,N} = \gen{a,b,Y_{0,N}}\leq \gen{Y_{0,N}}^X.
\end{equation}
By shift-invariance, we can now conclude that
\[Y_{-4,M}
= t^{-1}(Y_{-2,M+2})
\overset{(\ref{eqn:Y_{-2}-sub-Y_0^X})}\leq t^{-1}(\gen{Y_{0,M+2}}^X)
= \gen{Y_{-2,M}}^X
\overset{(\ref{eqn:Y_{-2}-sub-Y_0^X})}\leq \gen{Y_{0,M}}^X
.\]
By induction it follows that $Y_{-\infty,M}\leq \gen{Y_{0,M}}^X$.
\end{proof}

\begin{lemma}\label{lem:110}
If for all $n\in\ZZ$ we have 
\[
Y_{n-1,\infty}\leq\gen{Y_{n,\infty}}^X
\quad\text{ and }\quad
Y_{-\infty,n+1}\leq\gen{Y_{-\infty,n}}^X
\]
then there exists $M\in\NN$ such that $Y\leq\gen{Y_{0,M}}^X$.
\end{lemma}

\begin{proof}
This is an immediate consequence of \cref{finite cut}.
\end{proof}

\begin{prop}\label{[XY]=Y}
Suppose $[X,Y]=Y$. Then either $Y=1$, or there exists $n\in\ZZ$ such that
at least one of the following holds:
\begin{enumerate}
\item $Y_{n-1,\infty}\nleq\gen{Y_{n,\infty}}^X$.
\item $Y_{-\infty,n+1}\nleq\gen{Y_{-\infty,n}}^X$.
\end{enumerate}
\end{prop}

\begin{proof}
Suppose the claim is false. Then by \cref{lem:110}
there is $M\in\NN$ such that $Y\leq \gen{Y_{0,M}}^X$.
The group $Y_{0,M}$ is finite, so we can pick
a finite generating set $Y_{0,M}=\gen{z_1,\ldots,z_k}$.
Then for $1\leq i\leq k$, since $z_i\in Y=[X,Y]$, 
there are $\ell_i\in\NN$ and $y_{ij}\in Y$, $g_{ij}\in X$
for $1\leq j\leq \ell_i$ such that
\[ z_i = [g_{i1},y_{i1}]\cdots[g_{i\ell_i},y_{i\ell_i}].\]
Let 
\begin{align*}
n&:=\min\{\n(y_{ij}) \mid 1\leq i\leq k, 1\leq j\leq \ell_j \},\\
m&:=\max\{\m(y_{ij}) \mid 1\leq i\leq k, 1\leq j\leq \ell_j \}.
\end{align*}
Then $z_i\in [X,Y_{n,m}]$, hence $Y_{0,M}\subseteq [X,Y_{n,m}]$.
But then
\begin{equation} \label{eqn Ynm=1}
Y_{n,m} \leq Y \leq \gen{Y_{0,M}}^X \leq \gen{[X,Y_{n,m}]}^X.
\end{equation}
Since $X$ is locally nilpotent, \cref{A-notin-[AU]_U} implies $Y_{n,m}=1$.
Inserting this into \Cref{eqn Ynm=1} yields $Y=1$.
\end{proof} 

\begin{lemma}\label{normal-aux}
Suppose $Y\unlhd X$ and $|X:Y| =\infty$. Then the following hold.
\begin{enumerate}
\item If $Y_{n,\infty}\nleq\gen{Y_{n+1,\infty}}^X$ for some $n\in\ZZ$,
then $Y_{n,\infty}\unlhd X$.
\item If $Y_{-\infty,n}\nleq\gen{Y_{-\infty,n-1}}^X$ for some $n\in\ZZ$,
then $Y_{-\infty,n}\unlhd X$.
\end{enumerate}
\end{lemma}

\begin{proof}
We prove the first case, the second follows by a symmetric argument.
By \cref{MaxProposition5.7(3)} there is $y\in Y$ such that
$(Y,(t^k(y))_{k \in \ZZ})$ is a $\ZZ$-system.
Suppose now that there is $n\in\ZZ$ such that $Y_{n,\infty}\nleq\gen{Y_{n+1,\infty}}^X$.
Then as $Y$ is shift-invariant, we may assume that $\n(y)=n$, and so $y\in Y_{n,\infty}$
but $y\notin \gen{Y_{n+1,\infty}}^X$.

Suppose now that there is $m<n$ with $[x_m,y] \neq 1$.
Then there are integers $i_1 < \dots < i_s < 0$ and
exponents $e_1,\dots,e_s\in\ZZ_p^*$,
such that $[x_m,y] =t^{i_1}(y)^{e_1} \cdots t^{i_s}(y)^{e_s}$.
Applying $t^{-i_1}$ we get 
\[ [x_{m-2i_1},t^{-i_1}(y)] =y^{e_1} \cdot t^{i_2-i_1}(y)^{e_2} \cdots t^{i_s -i_1}(y)^{e_s} \]
and, since $-i_1$, $i_2-i_1$, \ldots, $i_s-i_1$ all are positive, we conclude
\[ y^{e_1} = [x_{m-2i_1}, t^{-i_1}(y)]
\cdot \left(t^{i_2-i_1}(y)^{e_2} \cdots t^{i_s -i_1}(y)^{e_s}\right)^{-1} \in \gen{Y_{n+1,\infty}}^X. \]
But $\n(y^{e_1})=\n(y)=n$, thus $Y_{n,\infty}=\gen{y^{e_1},Y_{n+1,\infty}} \leq\gen{Y_{n+1,\infty}}^X$,
contradicting the hypothesis. Therefore $[x_m,y] =1$ for all $m <n$.
Since $Y_{n,\infty} =\gen{t^i(y) \mid i\in\NN}$, we get
$X_{-\infty,n-1} \leq C_X(Y_{n,\infty})$ and so,
using that $Y\unlhd X$,
\[ [X,Y_{n,\infty}]
=[X_{n,\infty},Y_{n,\infty}]
\leq [X_{n,\infty},Y] \cap X_{n,\infty}
\leq Y \cap X_{n,\infty} = Y_{n,\infty}.
\qedhere
\]
\end{proof}

Thus we obtain the main result of this section:

\begin{prop} \label{normal}
Let $(X,(x_n)_{n \in \ZZ})$ be a $\ZZ$-system of prime order $p$.
Suppose $Y$ is a shift-invariant subgroup of $X$, with $|X:Y|=\infty$ and $[X,Y]=Y$.
Then there is $n\in\ZZ$ such that $Y_{n,\infty}\unlhd X$ or $Y_{-\infty,n}\unlhd X$,
where $Y_{n,\infty}:=Y\cap X_{n,\infty}$ and $Y_{-\infty,n}:=Y\cap X_{-\infty,n}$.
\end{prop}

\begin{proof}
This follows by first applying \cref{[XY]=Y}, then \cref{normal-aux}.
\end{proof}

\section{Infinite abelianization}

\begin{nota}
Let $G$ be a group.
Then let $G^{(0)}:=G$, let $G':=[G,G]$ be the \Defn{derived subgroup}
and for $k\in\NN$ let $G^{(k+1)}:=[G^{(k)},G^{(k)}]$.
\end{nota}

\begin{lemma}\label{non perfect}
Let $1\neq Y\unlhd X$ be shift-invariant. Then $[Y,Y]< Y$.
\end{lemma}

\begin{proof}
Suppose that $[Y,Y] =Y$. Then we have also $[X,Y]=Y$.
Since also $Y\neq1$, by \cref{[XY]=Y}
this implies that there exists $n\in \ZZ$ such that
$Y_{n-1,\infty}\nleq\gen{Y_{n,\infty}}^X$ or
$Y_{-\infty,n+1}\nleq\gen{Y_{-\infty,n}}^X$ holds.
Suppose that  $Y_{n-1,\infty}\nleq\gen{Y_{n,\infty}}^X$
(the other case is dealt with by a symmetric argument).

Let $N:=\gen{Y_{n,\infty}}^X$. Then $N\unlhd X$ and by what
we just said $Y_{n-1,\infty}\nleq N$, thus $Y\neq N$. On
the other hand, from $Y_{n,\infty}\leq Y\unlhd X$ it follows
that $N\unlhd Y$.

By \cref{decompose shift inv} there is $m \in \ZZ$ such that
\[ Y \overset{\ref{decompose shift inv}}= Y_{-\infty,m} Y_{n,\infty} \leq Y_{-\infty,m} N \leq Y, \]
hence $Y = Y_{-\infty,m} N$.
Choose $m\in\NN$ minimal with this property.
Then $Y_{-\infty,m}' \leq Y_{-\infty,m-1}$ by (ZS5) and so
\[ [Y,Y] = Y' \leq Y_{-\infty,m}'N \leq Y_{-\infty,m-1}N < Y,\]
a contradiction.
\end{proof}

\begin{cor} \label{derived index arbitrary}
For $k\in\NN$, we have $|X:X^{(k)}|\geq p^k$.
\end{cor}

\begin{proof}
The claim follows by induction on $k$, and the following observations:
$X^{(k)}$ is a characteristic subgroup of $X$, hence shift-invariant and normal.
Thus if $X^{(k)}\neq 1$, then $X^{(k+1)}< X^{(k)}$ by \cref{non perfect}.
And if $X^{(k)}= 1$, then $|X:X^{(k)}|=|X:1|=|X|=\infty$.
\end{proof}

\begin{lemma}[{\cite[Lemma 5.9]{MKS66}}] \label{lift-nilp-gens}
Let $G$ be a nilpotent group.
If $z_1,\ldots,z_\ell\in G$ satisfy $G/G'=\gen{z_1G',\ldots,z_\ell G'}$,
then $G=\gen{z_1,\ldots,z_\ell}$.
\end{lemma}

\begin{lemma} \label{inf-derived-index}
There is $k\in\NN$ such that $|X:X^{(k)}|=\infty$.
\end{lemma}

\begin{proof}
Suppose $|X:X^{(k)}|<\infty$ for all $k\in\NN$.
Choose $z_1,\ldots,z_\ell\in X$ such that $X/X'=\gen{z_1X',\ldots,z_\ell X'}$.
For $k\in\NN$, the groups $G_k:=X/X^{(k)}$ are finite $p$-groups and hence nilpotent.
Next observe that
\[
G_k/G_k'  \cong X/X' = \gen{z_1X',\ldots,z_\ell X'}
\]
implies that
\[
G_k/G_k' = \gen{\hat{z}_1G_k',\ldots,\hat{z}_\ell G_k'},
\]
where $\hat{z}_1:=z_1X^{(k)},\ldots,\hat{z}_\ell:=z_\ell X^{(k)}$.
Therefore, by \cref{lift-nilp-gens} we conclude
\[G_k=\gen{z_1X^{(k)},\ldots,z_\ell X^{(k)}}.\]
Now let $Z:=\gen{z_1,\ldots,z_\ell}\leq X$. Since $X$ is locally finite, $|Z|<\infty$.
It follows that $X=ZX^{(k)}$ for all $k\in\NN$, hence $|X:X^{(k)}|\leq |Z|$.
But this is a contradiction, as $|X:X^{(k)}|$ becomes arbitrarily large by \cref{derived index arbitrary}.
\end{proof}

\begin{lemma}[{\cite[5.2.6]{Rob96}}] \label{G finite}
A nilpotent group $G$ with $|G:G'| < \infty$ is finite.
\end{lemma}

\begin{lemma} \label{nilpotent-extension}
Let $G$ be a $p$-group, $N\unlhd G$ nilpotent of finite exponent and $|G:N| < \infty$.
Then $G$ is nilpotent of finite exponent.
\end{lemma}

\begin{proof}
We will assume $|G:N|=p$, the general case follows by induction on $|G:N|$.
Let
\[ Z_0:=1 \leq Z_1 := Z(N) \leq Z_2 \leq \dots \leq Z_n=N \]
be the upper central series of $N$. Then for all $i$, the 
$Z_i$ are characteristic in $N$ and hence normal in $G$.
Since $N$ has finite exponent, we can refine this series
to a series 
\[ W_0:=1 \leq W_1 \leq \dots \leq W_m = N \]
such that $W_i$ is normal in $G$ and $M_i:=W_i/W_{i-1}$ has exponent $p$ for all $i>0$.
In fact, since we refined a central series, the $M_i$ are elementary abelian $p$-groups,
in other words, vector spaces over a finite field of order $p$.

Let $x \in G \setminus N$. Since $N$ acts trivially on $M_i$, and since $|G:N|=p$,
it follows for all $i>0$ that $x$ induces an automorphism 
$x_i$ of order at most $p$ on the vector space $M_i$.
Since $x_i^p=1$, the linear map $x_i$ has a minimal polynomial dividing $t^p -1 =(t-1)^p$.

But then $[v,\underbrace{x_i, \ldots, x_i}_{p}]=1$ for all $v \in M_i$.
Hence we can refine the series in such a way that $G$ acts trivially on each factor.
Therefore $G$ is nilpotent, and since $N$ and $G/N$ have
finite exponent, the exponent of $G$ is also finite.
\end{proof}

\begin{remark}
Note that the condition that the exponent of $N$ is finite is essential. For example, let
$G$ be the injective limit of dihedral groups $(D_{2^n})_{n\geq 1}$, that is 
\[ G = \gen{ s,r_1,r_2,r_3,\dots \mid s^2=1=r_1^2,\ r_{n+1}^2=r_n,\
 r_ns=sr_n^{-1}\ \text{ for } n\geq 1}.
\]
Let $N$ the normal subgroup generated by the rotations $r_n$.
Then $N$ is an abelian $2$-group and $G/N$ has order $2$, but $[G,N]=N$.
\end{remark}

\begin{thm}\label{commutator group}
Let $(X,(x_n)_{n \in \ZZ})$ be a $\ZZ$-system of prime order $p$.
Then $X$ has infinite abelianization $X/X'$.
\end{thm}

\begin{proof}
For $k\in\NN$, let $G_k:=X/X^{(k)}$ and $H_k:=X^{(k)}/X^{(k+1)}$.
Since $G_0$ is trivial, \cref{inf-derived-index} implies that there is $k\in\NN$ such that
$|G_k|<\infty$ and $|G_{k+1}|=\infty$.
We have $X^{(k+1)} \leq X^{(k)} \leq X$ and therefore
\[ |G_{k+1}| = |X:X^{(k+1)}| = |X:X^{(k)}|\cdot |X^{(k)} : X^{(k+1)}| = |G_k| \cdot |H_k|.\]
Thus $|H_k| = \infty$.

Since $X^{(k)}$ is shift-invariant, by \cref{shift-inv-gens} it is generated
by the shifts of two elements $a,b\in X^{(k)}$, that is
\[ H_k = \gen{ t^m(a)X^{(k+1)},\ t^m(b)X^{(k+1)} \mid m\in\ZZ }.\]
Since $H_k$ is an abelian $p$-group, there is $n\in\NN$ such that these
generators all have orders dividing $p^n$.
Thus $H_k$ has finite exponent  and as $|G_{k+1}:H_k| = |G_k| < \infty$, the group $G_{k+1}$ is nilpotent by \cref{nilpotent-extension}.
But $G_{k+1}$ is infinite, so $G_{k+1}/G_{k+1}'\cong X/X'$ must also be infinite by
\cref{G finite}.
\end{proof}

\section{Nilpotency class 2} \label{sect hauptbeweis}

\begin{lemma} \label{commutator bilinear}
Let $Y\unlhd X$, $y,y'\in Y$ and $x\in X$. Then
$[yy',x] \in [y,x][y',x] [Y,X,X]$.
\end{lemma}

\begin{proof}
We have $[Y,X,X]\unlhd X$, hence
\[
[yy',x] = [y,x]^{y'} [y',x]
= [y,x]\underbrace{[y,x]^{-1} [y,x]^{y'}}_{=[[y,x],y']} [y',x]
\in [y,x][y',x] [Y,X,X].
\qedhere
\]
\end{proof}

\begin{lemma}\label{[YX]=[YXX]}
Let $Y\unlhd X$ be shift-invariant, and suppose $|X:Y| = \infty$.
Then $[Y,X,X] =[Y,X]$.
\end{lemma}

\begin{proof}
For $Y=1$ the claim is obvious, so we suppose $Y\neq 1$.
Since $[Y,X,X]\leq[Y,X]$, it suffices to show the reverse inclusion.

As $|X:Y| = \infty$, by \cref{MaxProposition5.7(3)} the shifts of any element $y\in Y^*$
of minimal width in $Y^*$ generate the group $Y$, which is abelian. Set $n:=\n(y)$ and $m:=\m(y)$.
Then $Y_{n+1,m}=1$ as $y$ is of minimal width in $Y^*$.
We will now show by induction on $N\geq n$ that $[y,x_N]\in[Y,X,X]$.
Indeed, for $n\leq N\leq m+1$, we have $[y,x_N]\in Y_{n+1,m}=1\leq[Y,X,X]$.

So suppose $N>m+1$, and $[y,x_N]\neq 1$. Since $[y,x_N]\in Y_{n+1,N-1}$, applying (ZS6)
to the $\ZZ$-system $(Y,t^k(y)_{k \in \ZZ})$ yields that there
are uniquely determined values $s\in\NN$, $i_1,\dots,i_s\in\NN$
and $\lambda_1,\dots,\lambda_s\in\ZZ_p^*$ such that
\begin{equation} \label{eqn:yxN=...}
0 < 2i_1  < \ldots < 2i_s \leq N-1-m
\quad\text{ and }\quad
[y,x_N] = t^{i_1}(y)^{\lambda_1} \cdots t^{i_s}(y)^{\lambda_s}.
\end{equation}
If $s>1$, then for $k=2,\ldots, s$, the preceding inequality together with $0<i_1<i_k$ implies 
\[ m+1\leq N-2i_k <
N+2i_1-2i_k = N - 2(i_k-i_1) < N, \]
hence by the induction hypothesis and by the shift-invariance of $[Y,X,X]$ we have
\begin{equation} \label{eqn:comms in YXX}
 [t^{i_k}(y),x_{N+2i_1}] =
t^{i_k}([y,x_{N+2i_1-2i_k}]) \in [Y,X,X].
\end{equation}
Applying \cref{commutator bilinear} repeatedly, we find
\begin{align*}
 [[y,x_N],x_{N+2i_1}]
&\overset{(\ref{eqn:yxN=...})}= [t^{i_1}(y)^{\lambda_1} \cdots t^{i_s}(y)^{\lambda_s},x_{N+2i_1}] \\
&\overset{\ref{commutator bilinear}}{\in} [t^{i_1}(y)^{\lambda_1},x_{N+2i_1}] \cdots [t^{i_s}(y)^{\lambda_s},x_{N+2i_1}] [Y,X,X] \\
&\overset{(\ref{eqn:comms in YXX})}{=} [t^{i_1}(y)^{\lambda_1},x_{N+2i_1}][Y,X,X] \\
&\overset{\ref{commutator bilinear}}{=} [t^{i_1}(y),x_{N+2i_1}]^{\lambda_1}[Y,X,X] \\
&= t^{i_1}([y,x_N])^{\lambda_1}[Y,X,X] \\
&= t^{i_1}([y,x_N]^{\lambda_1})[Y,X,X].
\end{align*}
Therefore $t^{i_1}([y,x_N]^{\lambda_1}) \in [Y,X,X]$.
But $[Y,X,X]$ is shift-invariant, hence we also have $[y,x_N]^{\lambda_1} \in [Y,X,X]$.
And $Y$ has prime exponent $p$, thus also $[y,x_N] \in [Y,X,X]$.
This concludes the proof of the claim that $[y,x_N] \in [Y,X,X]$ for all $N \geq n$.

A similar argument shows that $[y,x_N] \in  [Y,X,X]$ also holds for all $N < n$. 
But $[Y,X] =\gen{t^k([y,x_N]) \mid k,N \in \ZZ}^X$, therefore $[Y,X] = [Y,X,X]$.
\end{proof}

\begin{remark}
Suppose that $G$ and $V$ are groups and that $G$ acts on $V$
from the right by automorphisms. Then we define
\[ [V,G] := \gen{ v^{-1} \cdot v^g \mid g\in G, v\in V}. \]
This is a natural extension of the commutator
group notation, e.g.\ for $V\unlhd G$.
\end{remark}

\begin{lemma} \label{[VG]<V}
Let $G$ and $V$ be $p$-groups, with $G$ acting on $V$ by automorphisms.
If $V$ is finite and non-trivial,
then $[V,G]$ is a proper subgroup of $V$.
\end{lemma}

\begin{proof}
Let $\alpha: G\to\Aut(V)$ be the action homomorphism associated to the
action of $G$ on $V$. Since $V$ is finite, also $\Aut(V)$ is finite,
and hence $\widetilde{G}:=\alpha(G)$ is finite.
Clearly $[V,G] = [V,\widetilde{G}]$.
Form the semidirect product $K:=V\rtimes \widetilde{G}$.
Then $[V,\widetilde{G}]\leq[V,K]$.
Since $V\unlhd K$ we have $[V,K]\leq V$.
Moreover, $K$ is a finite $p$-group, and thus it is nilpotent.
Hence if $[V,K]=V$, then by \cref{A<=[AK] then A=1} we get $V=1$, a contradiction.
Thus 
\[ [V,G] = [V,\widetilde{G}] \leq [V,K]<V.
\qedhere\]
\end{proof}

\begin{lemma} \label{finite idx submod}
Let $Y\unlhd X$ be shift-invariant with
$[X,Y]\neq 1$.
Suppose there is $y \in Y^*$ 
such that $(Y,(t^k(y))_{k \in \ZZ})$ is a $\ZZ$-system.
Then $Y$ is elementary abelian, and for $m:=\m(y)$,
the group $M:=Y_{-\infty,m}=Y\cap X_{-\infty,m}$ is an
$\FF_p X_{-\infty,m}$-module, and $M_0:=[M,X_{-\infty,m}]$ is a
proper, non-trivial submodule of finite index.
\end{lemma}

\begin{proof}
The group $Y$ is elementary abelian by \cref{abelian<=>single-shift-inv},
hence so is $M$.
As $Y\unlhd X$, the group $M$ is an $\FF_p X_{-\infty,m}$-module.
We compute
\[
[Y,X]
= \bigcup_{k\in\NN} [Y_{-\infty,m+2k},X_{-\infty,m+2k}]
= \bigcup_{k\in\NN} t^k(M_0).
\]
The hypothesis states $[X,Y]\neq 1$, so we must have $M_0\neq 1$.
Moreover $y\in M$, but
\[ M_0
= [Y_{-\infty,m},X_{-\infty,m}]
\leq [X_{-\infty,m},X_{-\infty,m}]
\leq X_{-\infty,m-1},
\]
and $y\notin X_{-\infty,m-1}$, hence $y\notin M_0$.
We conclude that $M_0\neq M$, i.e.\ $M_0$ is a proper, non-trivial submodule.

Since $M=\gen{t^{-k}(y)\mid k\in\NN}$, we may also regard $M$ as an $\FF_p[t^{-1}]$-module,
which  is generated by $y\in M$. Hence it is a free $\FF_p[t^{-1}]$-module of rank $1$.
Now $M_0$ is a proper non-trivial $\FF_p[t^{-1}]$-submodule of $M$,
thus $M_0$ must have finite index in $M$.
\end{proof}

We are now ready to prove our main theorem.

\begin{proof}[Proof of \cref{hauptsatz}]
Set $Y:=[X,X,X]$. Our goal is to prove $Y=1$.
Clearly $Y\unlhd X$ and also $Y\unlhd X'$ hold.
By \cref{commutator group} we have $|X:X'|=\infty$.
We thus may apply \cref{[YX]=[YXX]} for $X'$, which yields
\[
Y \overset{\mathrm{def.}}= [X',X] \overset{\ref{[YX]=[YXX]}}= [X',X,X]
\overset{\mathrm{def.}}= [Y,X].
\]
In addition, $Y\leq X'$ and $|X:X'|=\infty$ imply $|X:Y|=\infty$.
Therefore \cref{normal} is applicable, and proves that there is
$n \in \ZZ$ such that $Y_{n,\infty}\unlhd X$ or $Y_{-\infty,n}\unlhd X$.
We may assume (up to a relabeling of the generators of $X$) without loss of generality
that the first case holds.

We proceed by assuming that $Y\neq1$ and derive a contradiction.
By \cref{MaxProposition5.7(3)} there is $y \in Y^*$ with $\n(y) =n$
and $Y_{n,\infty}=\gen{t^k(y)\mid k\in\NN}$.
Let \[ N := Y_{n+2,\infty}, \quad m:=\m(y), \quad Y_0:=[Y/N, X_{-\infty,m}], \]
where we regard $Y/N$ as an $\FF_p X$-module, which is feasible
since $Y\unlhd X$ and also 
\[N=t(Y_{n,\infty})\unlhd t(X) = X. \]
We claim that $Y_0$ is an $\FF_p X$-submodule of $Y/N$.
Indeed, we have
\[[X_{-\infty,m},X] \leq X' \leq C_X(Y),
\quad\text{ implying }\quad
X_{-\infty,m}^g \subseteq X_{-\infty,m} C_X(Y)\]
for all $g\in X$.
Moreover, from $Y=Y^g$ and $[a,bc] = [a,c][a,b]^c$ it follows that
\[
[Y, X_{-\infty,m}]^g
= [Y^g, X_{-\infty,m}^g]
\leq [Y, X_{-\infty,m}C_X(Y)] = [Y, X_{-\infty,m}].\]
Hence $Y_0$ is indeed an $\FF_p X$-submodule of $Y/N$.

By \cref{finite idx submod}, we have $1 < |M : M_0| < \infty$ for
\[M:=Y_{-\infty,m}, \quad M_0 := [M, X_{-\infty,m}]. \]
Since $Y/N$ is an $\FF_p X$-module, it is also an $X_{-\infty,m}$-module.
In fact $Y/N$ and $M$ are isomorphic as $X_{-\infty,m}$-modules:
Indeed, $Y$ is the inner direct product of $N$ and $M$, thus we get the isomorphism
\[ M \to Y/N,\  g \mapsto g N.\]
This isomorphism maps $M_0$ to $Y_0$, and so
\cref{finite idx submod} implies $1 < |Y/N : Y_0| < \infty$.

Therefore
$A:=(Y/N)/Y_0$ is a non-trivial, finite $p$-group on which $X$
acts by automorphisms, and so \cref{[VG]<V} implies $[A,X]<A$.
Yet earlier on we proved $[Y,X]=Y$, which implies
\[ [A,X] = [(Y/N)/Y_0, X] = (Y/N)/Y_0 = A. \]
But this is a contradiction. Hence our initial assumption that $Y\neq 1$
was wrong, and so $Y$ is trivial.
Since by definition $Y=[X,X,X]$, this completes the claim.
\end{proof}

\paragraph{Acknowledgments.}

At an early stage of this project, we had proven a weaker result, namely that
\emph{nilpotent} $\ZZ$-systems of order 2 are nilpotent of class 2.
We would like to point out that this preliminary result has been proved independently by
Bettina Wilkens.
The idea of making use of $\FF_p[t]$-modules in this context, which we adopted
for the proof of \cref{hauptsatz}, is due to her.

We are grateful to Barbara Baumeister, Maximilian Parr and Richard Weiss
for careful proofreading and helpful comments.
We also thank Pierre-Emmanuel Caprace and the anonymous referee for useful suggestions.

The research for this paper was undertaken
while the first author was on a post-doc position at UCLouvain,
in the research group of Pierre-Emmanuel Caprace, and
funded by ERC grant \#278469.
The project was also partially supported by DFG grant MU 1281/5-4.

\end{document}